\documentclass[reqno]{amsart}

\usepackage[utf8]{inputenc}
\usepackage[T1]{fontenc}
\usepackage{lmodern}
\usepackage[top=3cm, bottom=3cm, left=3cm, right=3cm]{geometry}
\usepackage{amsmath, amssymb, amsthm, mathtools}
\usepackage{hyperref}
\usepackage{xcolor}
\usepackage{cancel}
\usepackage[normalem]{ulem}
\usepackage{enumitem}
\usepackage{fancyhdr}
\usepackage{booktabs}

\newtheorem{theorem}{Theorem}[section]
\newtheorem{definition}[theorem]{Definition}
\newtheorem{corollary}[theorem]{Corollary}
\newtheorem{proposition}[theorem]{Proposition}
\newtheorem{lemma}[theorem]{Lemma}

\newtheorem{remark}[theorem]{Remark}

\newcommand{\ii}{\mathfrak{i}}

\newcommand{\g}{\mathfrak{g}}

\newcommand{\fn}{\mathfrak n}

\DeclareMathOperator{\Lie}{Lie}
\DeclareMathOperator{\Aut}{Aut}

\DeclareMathOperator{\GL}{GL}

\DeclareMathOperator{\Der}{Der}

\DeclareMathOperator{\nrad}{nilrad}

\newcommand{\tcdot}{\mathchoice
  {\raise-0.20ex\hbox{$\displaystyle\tilde{\cdot}$}}
  {\raise-0.20ex\hbox{$\textstyle\tilde{\cdot}$}}
  {\raise-0.15ex\hbox{$\scriptstyle\tilde{\cdot}$}}
  {\raise-0.10ex\hbox{$\scriptscriptstyle\tilde{\cdot}$}}}
\newcommand{\tcirc}{\mathchoice
  {\raise-0.20ex\hbox{$\displaystyle\tilde{\circ}$}}
  {\raise-0.20ex\hbox{$\textstyle\tilde{\circ}$}}
  {\raise-0.15ex\hbox{$\scriptstyle\tilde{\circ}$}}
  {\raise-0.10ex\hbox{$\scriptscriptstyle\tilde{\circ}$}}}

\newcommand{\id}{\mathrm{id}}

\begin{document}

\title[Simply connected simple Lie skew braces]{On Simply Connected Simple Lie Skew Braces with Nilpotent Multiplicative Group}

\author{Marco Damele}
\address{(Marco Damele) Dipartimento di Matematica \\
         Universit\`a di Cagliari (Italy)}
\email{m.damele4@studenti.unica.it}

\author{Andrea Loi}
\address{(Andrea Loi) Dipartimento di Matematica \\
         Universit\`a di Cagliari (Italy)}
\email{loi@unica.it}

\thanks{
The authors are supported by INdAM and GNSAGA - Gruppo Nazionale per le Strutture Algebriche, Geometriche e le loro Applicazioni; the research was also supported by ProBiki of Fondazione di Sardegna.}

\subjclass[2000]{16T25, 17B05, 17B30, 22E60, 22E46}
\keywords{Lie skew brace; post-Lie algebra; affine action; simple Lie skew brace; nilpotent Lie group}

\begin{abstract}
We prove that a simply connected simple Lie skew brace with nilpotent
multiplicative Lie group must be one-dimensional and abelian.
Equivalently, if $(G,\cdot,\circ)$ is a simply connected Lie skew brace with
nilpotent multiplicative Lie group and
$\dim G>1$,
then $(G,\cdot,\circ)$ is not simple.
Thus, in the simply connected Lie setting, nilpotency of the multiplicative
group is incompatible with simplicity in every dimension greater than one.
 The proof is carried out at the post-Lie algebra level.
First, if the additive Lie algebra is solvable, then its nilradical is
automatically an ideal of the associated post-Lie algebra. Second, when both
Lie algebras underlying an integrable post-Lie structure are nilpotent, one
always obtains a proper post-Lie ideal with trivial quotient. To pass from
infinitesimal ideals to global ideals of the Lie skew brace, we show that
trivial post-Lie quotients give rise to homomorphisms onto abelian trivial Lie
skew braces, whose kernels yield connected closed ideals.
\end{abstract}

\maketitle

\section{Introduction}\label{sec:intro}

Skew braces, introduced by Guarnieri and Vendramin
\cite{Guarnieri2017}, provide an algebraic framework for the study of
set-theoretic solutions of the Yang--Baxter equation. Their Lie-theoretic
counterpart is formed by \emph{Lie skew braces}, namely smooth manifolds
endowed with two compatible Lie group structures. Connected Lie skew braces
admit a natural interpretation in terms of simply transitive affine actions on
Lie groups and, infinitesimally, give rise to post-Lie algebra structures;
see \cite{DameleLoi2026,BaiGuoShengTang2023,Burde2012affineaction}. Thus, Lie
skew braces provide a global setting for the study of integrable post-Lie
algebras.

A natural problem is to understand how restrictive simplicity is in this
context. In the finite theory, simple skew braces may exhibit a rather flexible
behaviour: they need not arise from simple groups, and phenomena with no direct
analogue in the Lie setting may occur; see, for instance,
\cite{Bachiller2018,Byott2024simple}. By contrast, the available results for
Lie skew braces point towards a substantially more rigid picture.

In our recent work \cite{DameleLoiCompactSimple}, we studied compact connected
simple Lie skew braces. We proved that, apart from the one-dimensional trivial
brace on $S^1$, simplicity forces both underlying Lie groups to be simple and
the brace to be trivial or almost trivial. The purpose of the present paper is
to establish a complementary rigidity result under a different structural
assumption, namely nilpotency of the multiplicative Lie group.

Throughout the paper, an ideal of a connected Lie skew brace means a connected
closed Lie subgroup which is normal for both group structures and invariant
under the associated $\lambda$-action. A Lie skew brace is called simple if it
has no nontrivial proper ideals in this sense.

Our main result is the following.

\begin{theorem}\label{thm:main}
Let $(G,\cdot,\circ)$ be a simply connected simple Lie skew brace. Assume that
the multiplicative Lie group $(G,\circ)$ is nilpotent. Then both underlying Lie
groups are one-dimensional and abelian. In particular,
\[
(G,\cdot,\circ)\cong (\mathbb R,+,+).
\]
\end{theorem}

Equivalently, every simply connected Lie skew brace of dimension greater than
one whose multiplicative Lie group is nilpotent admits a nontrivial proper
connected closed ideal. Hence, in the simply connected Lie setting, nilpotency
of the multiplicative group is incompatible with simplicity in every dimension
greater than one.

The proof combines a global integration argument with two infinitesimal
constructions for post-Lie algebras. Let
$(\g,[\, ,\,]_{\cdot},[\, ,\,]_{\circ},\triangleright)$
be the post-Lie algebra associated with a connected Lie skew brace. The first
ingredient concerns the case in which the additive Lie algebra
$(\g,[\, ,\,]_{\cdot})$ is solvable. We prove that its nilradical is
automatically an ideal of the post-Lie algebra. Moreover, the corresponding
quotient post-Lie algebra is trivial. In the simply connected group setting,
such a trivial quotient integrates to a homomorphism onto an abelian trivial
Lie skew brace, and the identity component of its kernel is a connected closed
ideal of the original Lie skew brace.

The second ingredient treats the case in which both Lie algebras underlying an
integrable post-Lie structure are nilpotent. Using the triangular form arising
from simply transitive NIL-affine actions, we construct a proper post-Lie ideal
of codimension one with trivial quotient whenever
$\dim\g>1$.
This produces a nontrivial proper connected closed ideal at the group level.

To prove Theorem~\ref{thm:main}, we use the structural results of
\cite{DameleLoi2026}, according to which nilpotency of $(G,\circ)$ implies
solvability of $(G,\cdot)$. If the Lie skew brace is simple, the nilradical
argument then forces $(G,\cdot)$ to be nilpotent as well. Thus both Lie
algebras of the associated post-Lie structure are nilpotent. The
nilpotent--nilpotent argument yields a proper ideal unless the dimension is
one, and the one-dimensional case is necessarily trivial.

The simply connected hypothesis is essential in the passage from infinitesimal
quotients to closed ideals at the group level. Indeed, in the non-simply
connected case, a trivial post-Lie quotient may fail to integrate to a
homomorphism onto a Lie group quotient because of a period obstruction; see
Remark~\ref{rem:non-simply-connected-obstruction}.

The paper is organized as follows. In
Section~\ref{sec:preliminaries}, we recall the necessary material on Lie skew
braces and post-Lie algebras, and we show that trivial post-Lie quotients
integrate to connected closed ideals in the simply connected setting. In
Section~\ref{sec:auxiliary}, we establish the two infinitesimal ideal
constructions used in the proof. Finally, in
Section~\ref{sec:proof-main}, we prove Theorem~\ref{thm:main}.

\section{Preliminaries on Lie skew braces and post-Lie algebras}\label{sec:preliminaries}

We begin by recalling the notions that will be needed throughout the paper.

\subsection{Lie skew braces}

\begin{definition}
A \emph{Lie skew brace} $(G, \cdot, \circ)$ consists of  a smooth manifold $G$ endowed with two Lie group
structures, denoted by $(G,\cdot)$ and $(G,\circ)$ such that
\[
g\circ (h\cdot k)=(g\circ h)\cdot g^{-1}\cdot (g\circ k)
\qquad \forall\,g,h,k\in G,
\]
where the inverse and the products on the right-hand side are taken with
respect to the law $\cdot$. 
\end{definition}

\begin{definition}
A Lie skew brace $(G, \cdot, \circ)$ is \emph{trivial} if $(G,\cdot)=(G,\circ)$.
\end{definition}

Associated with a Lie skew brace is the map
\[
\lambda:(G,\circ)\longrightarrow \Aut(G,\cdot),
\qquad
\lambda_g(h):=g^{-1}\cdot (g\circ h),
\]
which is a smooth group homomorphism; see
\cite[Lemma~2.4]{DameleLoi2026}. The brace identity may equivalently be written
as
\[
g\circ h=g\cdot \lambda_g(h)
\qquad \forall\, g,h\in G.
\]

\begin{definition}\label{def:ideal-LSB}
Let $(G,\cdot,\circ)$ be a connected Lie skew brace. A connected Lie subgroup
$I\leq G$
is called an \emph{ideal} of $(G,\cdot,\circ)$ if
\begin{enumerate}[label=\textup{(\roman*)}]
    \item $I$ is closed in $G$;
    \item $I\trianglelefteq (G,\cdot)$;
    \item $I\trianglelefteq (G,\circ)$;
    \item $\lambda_g(I)=I$ for every $g\in G$.
\end{enumerate}
The Lie skew brace is called \emph{simple} if its only ideals are $\{e\}$ and
$G$.
\end{definition}

\subsection{Post-Lie algebras}

Let $(G,\cdot,\circ)$ be a connected Lie skew brace and set
\[
\g:=T_eG.
\]
Then the two Lie group structures induce two Lie algebra structures on $\g$,
namely
\[
\Lie(G,\cdot)=(\g,[\, ,\,]_{\cdot}),
\qquad
\Lie(G,\circ)=(\g,[\, ,\,]_{\circ}).
\]
Differentiating the lambda-action at the identity yields a bilinear product
\[
\triangleright:\g\times \g\to \g,
\qquad
x\triangleright y:=\lambda_{*e}(x)(y),
\]
and one obtains a post-Lie algebra structure
$(\g,[\, ,\,]_{\cdot},[\, ,\,]_{\circ},\triangleright)$;
see \cite{BaiGuoShengTang2023,Burde2012affineaction}.

\begin{definition}\label{def:post-Lie}
A \emph{post-Lie algebra} is a quadruple
$(\g,[\, ,\,]_{\cdot},[\, ,\,]_{\circ},\triangleright)$,
where $(\g,[\, ,\,]_{\cdot})$ and $(\g,[\, ,\,]_{\circ})$ are Lie algebras on
the same finite-dimensional real vector space $\g$, and
\[
\triangleright:\g\times \g\to \g
\]
is a bilinear map such that, for all $x,y,z\in \g$,
\begin{align*}
[x,y]_{\circ}-[x,y]_{\cdot}&=x\triangleright y-y\triangleright x,\\
x\triangleright [y,z]_{\cdot}&=[x\triangleright y,z]_{\cdot}
+[y,x\triangleright z]_{\cdot},\\
[x,y]_{\circ}\triangleright z&=
x\triangleright (y\triangleright z)-y\triangleright (x\triangleright z).
\end{align*}
\end{definition}

For $x\in \g$, we write
\[
L_x(y):=x\triangleright y.
\]
The second post-Lie identity says precisely that each $L_x$ is a derivation of
$(\g,[\, ,\,]_{\cdot})$.

\begin{definition}
A \emph{post-Lie algebra} 
$(\g,[\, ,\,]_{\cdot},[\, ,\,]_{\circ},\triangleright)$ is \emph{trivial}
if $[\, ,\,]_{\cdot}=[\, ,\,]_{\circ}=\triangleright=0$.
\end{definition}

\begin{definition}\label{def:ideal-postlie}
Let
$(\g,[\, ,\,]_{\cdot},[\, ,\,]_{\circ},\triangleright)$
be a post-Lie algebra. A vector subspace $\ii\subseteq \g$ is an
\emph{ideal} of the post-Lie algebra if
\[
[\g,\ii]_{\cdot}\subseteq \ii,\qquad
\g\triangleright \ii\subseteq \ii,\qquad
\ii\triangleright \g\subseteq \ii.
\]
The post-Lie algebra is called \emph{simple} if its only ideals are $0$ and
$\g$.
\end{definition}

\begin{remark}\label{rem:circ-consequence}
The bracket $[\, ,\,]_{\circ}$ does not explicitly appear in
Definition~\ref{def:ideal-postlie}. However, if $\ii\subseteq \g$ satisfies
\[
[\g,\ii]_{\cdot}\subseteq \ii,\qquad
\g\triangleright \ii\subseteq \ii,\qquad
\ii\triangleright \g\subseteq \ii,
\]
then automatically
\[
[\g,\ii]_{\circ}\subseteq \ii.
\]
Indeed, for $x\in \g$ and $a\in \ii$,
\[
[x,a]_{\circ}=[x,a]_{\cdot}+x\triangleright a-a\triangleright x,
\]
and every term on the right-hand side lies in $\ii$.
\end{remark}

For any Lie group $H$, we denote by $H_e$ the identity component of $H$.

\subsection{Trivial post-Lie quotients and closed ideals}

The following result is the basic device that allows us to pass from
infinitesimal quotients to connected closed ideals at the group level.

\begin{proposition}\label{prop:trivial-quotient-closed-ideal}
Let $(G,\cdot,\circ)$ be a simply connected Lie skew brace, and let
$(\g,[\, ,\,]_{\cdot},[\, ,\,]_{\circ},\triangleright)$
be its associated post-Lie algebra. Let $\ii\subseteq \g$ be an ideal of the
post-Lie algebra, and assume that the quotient post-Lie algebra $\g/\ii$ is
trivial. Then there exist a connected simply connected abelian Lie group $A$
and a Lie skew brace homomorphism
\[
\varphi:(G,\cdot,\circ)\to (A,+,+)
\]
such that
\[
\ker(\varphi_{*e})=\ii.
\]
Moreover, if
\[
K:=\ker\varphi
\qquad\text{and}\qquad
I:=K_e,
\]
then $I$ is a connected closed ideal of $(G,\cdot,\circ)$ and
$\Lie(I)=\ii$.
In particular, if $0\neq \ii\neq \g$, then $(G,\cdot,\circ)$ is not simple.
\end{proposition}
\begin{proof}
Set
\[
\mathfrak a:=\g/\ii,
\]
and let
\[
q:\g\to\mathfrak a
\]
be the quotient map. Since the quotient post-Lie algebra is trivial,
$\mathfrak a$ is an abelian Lie algebra and the induced post-Lie product
vanishes.

Since $(G,\cdot)$ is simply connected, the Lie algebra homomorphism
\[
q:(\g,[\, ,\,]_{\cdot})\to(\mathfrak a,0)
\]
integrates uniquely to a Lie group homomorphism
\[
\varphi:(G,\cdot)\to A:=(\mathfrak a,+).
\]
By construction,
\[
\varphi_{*e}=q,
\qquad
\ker(\varphi_{*e})=\ii.
\]

It remains to show that $\varphi$ is also a homomorphism with respect to the
$\circ$-structure. Since $\ii$ is a post-Lie ideal, each $\lambda_{g*e}$
preserves $\ii$, and therefore induces a linear endomorphism
\[
\rho(g):\mathfrak a\to\mathfrak a,
\qquad
\rho(g)(q(x)):=q(\lambda_{g*e}(x)).
\]
Thus we obtain a smooth representation
\[
\rho:(G,\circ)\to \GL(\mathfrak a).
\]
Its differential at the identity is given by
\[
\rho_{*e}(x)(q(y))=q(x\triangleright y)=0,
\qquad \forall\,x,y\in\g,
\]
because the induced post-Lie product on $\mathfrak a$ is trivial. Since
$(G,\circ)$ is connected, it follows that $\rho$ is trivial. Therefore
\[
q\circ \lambda_{g*e}=q
\qquad \forall\,g\in G.
\]

Now the Lie group homomorphisms
\[
\varphi\circ\lambda_g,\qquad \varphi
\]
from $(G,\cdot)$ to $A$ have the same differential at the identity. Since
$(G,\cdot)$ is simply connected, they coincide. Hence
\[
\varphi\circ\lambda_g=\varphi
\qquad \forall\,g\in G.
\]

Using
\[
g\circ h=g\cdot\lambda_g(h),
\]
we obtain
\[
\varphi(g\circ h)
=
\varphi(g\cdot\lambda_g(h))
=
\varphi(g)+\varphi(\lambda_g(h))
=
\varphi(g)+\varphi(h).
\]
Thus
\[
\varphi:(G,\cdot,\circ)\to(A,+,+)
\]
is a Lie skew brace homomorphism.

Set
\[
K:=\ker\varphi.
\]
Since $\varphi$ is a Lie group homomorphism for both structures, $K$ is a
closed normal subgroup of both $(G,\cdot)$ and $(G,\circ)$.

We now prove that $K$ is $\lambda$-stable. Let $g\in G$ and $k\in K$. Then
\[
\varphi(g\circ k)=\varphi(g)+\varphi(k)=\varphi(g).
\]
On the other hand,
\[
g\circ k=g\cdot\lambda_g(k),
\]
and therefore
\[
\varphi(g\circ k)
=
\varphi(g)+\varphi(\lambda_g(k)).
\]
Hence
\[
\varphi(\lambda_g(k))=0,
\]
so $\lambda_g(k)\in K$. Thus
\[
\lambda_g(K)\subseteq K
\qquad \forall\,g\in G.
\]
Applying the same argument to $g^{-1}$ gives equality. Hence $K$ is
$\lambda$-stable.

Therefore $K$ is a closed ideal of $(G,\cdot,\circ)$, and so is its identity
component
\[
I:=K_e.
\]
Finally, since $K$ is a closed Lie subgroup,
\[
\Lie(K)=\ker(\varphi_{*e})=\ii.
\]
Since $I=K_e$, we also have
\[
\Lie(I)=\Lie(K)=\ii.
\]
If $0\neq \ii\neq \g$, then $I$ is nontrivial and proper. This completes the
proof.
\end{proof}
\begin{remark}\label{rem:non-simply-connected-obstruction}\rm
The simply connected assumption in Proposition~\ref{prop:trivial-quotient-closed-ideal}
is used only to integrate the quotient map
\[
q:\g\to \g/\ii
\]
to a globally defined homomorphism from $(G,\cdot)$ to the abelian Lie group
$(\g/\ii,+)$.

Without this assumption, one has to pass to the universal cover
\[
p:\widetilde G\to (G,\cdot).
\]
The homomorphism
\[
q\circ p_{*e}:\Lie(\widetilde G)\to \g/\ii
\]
integrates to a homomorphism
\[
\widetilde\varphi:\widetilde G\to \g/\ii.
\]
It descends to $G$ only after quotienting by the period subgroup
\[
\Pi:=\widetilde\varphi(\ker p)\subseteq \g/\ii.
\]
Thus one would like to set
\[
A:=(\g/\ii)/\Pi.
\]
However, this construction yields a Lie group only if $\Pi$ is closed,
equivalently discrete. This condition is not automatic: a finitely generated
subgroup of a real vector space need not be discrete. For instance,
\[
\mathbb Z+\sqrt2\,\mathbb Z\subset \mathbb R
\]
is finitely generated but dense.

Therefore, in the non-simply connected case, a post-Lie ideal with trivial
quotient need not integrate to a connected closed ideal of the Lie skew brace.
The simply connected hypothesis removes this period obstruction, since then
$\ker p=0$ and hence $\Pi=0$.
\end{remark}

\subsection{One-dimensional Lie skew braces}

We record the following simple observation, which will be used at the end of the
proof of the main theorem.

\begin{lemma}\label{lem:1dim-trivial}
Every connected one-dimensional Lie skew brace is trivial.
\end{lemma}

\begin{proof}
Let $(G,\cdot,\circ)$ be a connected one-dimensional Lie skew brace. By
definition, there exists a continuous homomorphism
\[
\lambda:(G,\circ)\to \Aut(G,\cdot)
\]
such that
\[
g\circ h = g\cdot \lambda_g(h)
\qquad \forall\, g,h\in G.
\]

Since every connected one-dimensional Lie group is abelian, both
$(G,\cdot)$ and $(G,\circ)$ are abelian.

We distinguish two cases according to the underlying manifold $G$.

\smallskip

\noindent
\emph{Case 1: $G$ is diffeomorphic to $S^1$.}
Then $(G,\cdot)\cong S^1$, and hence
\[
\Aut(G,\cdot)\cong \Aut(S^1)\cong \{\pm 1\},
\]
which is discrete. Since $(G,\circ)$ is connected and $\lambda$ is continuous,
the image of $\lambda$ is connected, hence reduced to a point. As
$\lambda_e=\id_G$, it follows that $\lambda_g=\id_G$ for all $g\in G$.
Therefore
\[
g\circ h=g\cdot h
\qquad \forall\, g,h\in G.
\]

\smallskip

\noindent
\emph{Case 2: $G$ is diffeomorphic to $\mathbb R$.}
Then $(G,\cdot)\cong (\mathbb R,+)$. After identifying $(G,\cdot)$ with
$(\mathbb R,+)$, every automorphism of $(G,\cdot)$ is of the form
\[
y\longmapsto a\,y
\qquad (a\in \mathbb R^\times).
\]
Thus there exists a continuous map
\[
a:G\to \mathbb R^\times
\]
such that
\[
\lambda_x(y)=a(x)y
\qquad \forall\,x,y\in G.
\]
Since $(G,\circ)$ is connected and $\lambda_e=\id$ it follows that 
\[
a(x)>0,\ 
\qquad \forall\,x\in G.
\]

The brace identity becomes
\[
x\circ y = x + a(x)y, \ 
\qquad \forall\,x,y\in \mathbb R.
\]
But $(G,\circ)$ is also a connected one-dimensional Lie group, hence abelian.
Therefore
\[
x+a(x)y = y+a(y)x,\ 
\forall\,x,y\in \mathbb R.
\]
Rearranging, we obtain
\[
x\bigl(1-a(y)\bigr)=y\bigl(1-a(x)\bigr), \ \forall\,x,y\in \mathbb R.
\]
For $x,y\neq 0$, this implies
\[
\frac{1-a(x)}{x}=\frac{1-a(y)}{y}.
\]
Hence there exists a constant $c\in\mathbb R$ such that
\[
a(x)=1-cx, \ \forall\,x\in\mathbb R.
\]
Since $a(x)>0$ for all $x\in\mathbb R$, necessarily $c=0$. Thus $a(x)=1$ for
all $x$, so $\lambda_x=\id$ for every $x\in G$. Therefore again
\[
x\circ y=x+y
\qquad \forall\,x,y\in G.
\]

In both cases, $\lambda$ is trivial, and hence the two group laws coincide.
Therefore the Lie skew brace is trivial.
\end{proof}
\section{Auxiliary post-Lie ideal constructions}\label{sec:auxiliary}

In this section we establish the two main infinitesimal ingredients used in the
proof of Theorem~\ref{thm:main}. The first applies when the additive Lie
algebra is solvable and shows that its nilradical is automatically compatible
with the post-Lie structure. The second concerns the case where both Lie
algebras are nilpotent and yields a proper post-Lie ideal with trivial
quotient. Together with Proposition~\ref{prop:trivial-quotient-closed-ideal},
these results will allow us to produce connected closed ideals at the level of
Lie skew braces.

\subsection{The solvable additive case}\label{sec:solvable-additive}

We begin with the solvable additive case. The key point is that, for a
finite-dimensional real post-Lie algebra whose additive Lie algebra is
solvable, the nilradical is automatically stable under all post-Lie operators.
This provides a natural source of post-Lie ideals.

\begin{proposition}\label{prop:nilrad-postlie}
Let
$(\g,[\, ,\,]_{\cdot},[\, ,\,]_{\circ},\triangleright)$
be a finite-dimensional real post-Lie algebra. Assume that the additive Lie
algebra
$(\g,[\, ,\,]_{\cdot})$
is solvable, and let
$\fn:=\nrad(\g,[\, ,\,]_{\cdot})$
be its nilradical. Then $\fn$ is an ideal of the post-Lie algebra.
\end{proposition}

\begin{proof}
Since $\fn$ is the nilradical of $(\g,[\, ,\,]_{\cdot})$, it is in particular
an ideal of the Lie algebra $(\g,[\, ,\,]_{\cdot})$. Hence
\[
[\g,\fn]_{\cdot}\subseteq \fn.
\]

It remains to prove
\[
\g\triangleright \fn\subseteq \fn
\qquad\text{and}\qquad
\fn\triangleright \g\subseteq \fn.
\]

For every $x\in \g$, the operator
\[
L_x:\g\to \g,\qquad L_x(y)=x\triangleright y,
\]
is a derivation of the solvable Lie algebra $(\g,[\, ,\,]_{\cdot})$.
By Jacobson's theorem, every derivation of a finite-dimensional
solvable Lie algebra has image contained in its nilradical; see
\cite[p.~52, Corollary~2]{Jacobson}.
Therefore
\[
L_x(\g)\subseteq \fn
\qquad \forall\, x\in \g,
\]
that is,
\[
\g\triangleright \g\subseteq \fn.
\]
In particular,
\[
\g\triangleright \fn\subseteq \fn,
\qquad
\fn\triangleright \g\subseteq \fn.
\]
Thus $\fn$ is an ideal of the post-Lie algebra.
\end{proof}

As an immediate consequence, simplicity forces a strong restriction on the
additive Lie algebra.

\begin{corollary}\label{cor:simple-solvable-additive}
Let
$(\g,[\, ,\,]_{\cdot},[\, ,\,]_{\circ},\triangleright)$
be a simple finite-dimensional real post-Lie algebra. If the additive Lie
algebra
$(\g,[\, ,\,]_{\cdot})$
is solvable, then it is nilpotent.
\end{corollary}

\begin{proof}
Let
$\fn:=\nrad(\g,[\, ,\,]_{\cdot})$.
By Proposition~\ref{prop:nilrad-postlie}, $\fn$ is an ideal of the post-Lie
algebra. Since the post-Lie algebra is simple, either
\[
\fn=0
\qquad\text{or}\qquad
\fn=\g.
\]
If $(\g,[\, ,\,]_{\cdot})$ is solvable and nonzero, then its nilradical is
nonzero. Hence $\fn\neq 0$, and simplicity forces $\fn=\g$. Therefore
$(\g,[\, ,\,]_{\cdot})$ is nilpotent.
\end{proof}

Passing from the infinitesimal setting back to connected Lie skew braces, we
obtain the following global consequence.

\begin{corollary}\label{cor:simple-brace-solvable-additive}
Let $(G,\cdot,\circ)$ be a simply connected simple Lie skew brace, and let
$(\g,[\, ,\,]_{\cdot},[\, ,\,]_{\circ},\triangleright)$
be its associated post-Lie algebra. If the additive Lie group $(G,\cdot)$ is
solvable, then it is nilpotent.
\end{corollary}

\begin{proof}
Assume that $(G,\cdot)$ is solvable. Then the additive Lie algebra
$(\g,[\, ,\,]_{\cdot})$
is solvable. Let
\[
\fn:=\nrad(\g,[\, ,\,]_{\cdot}).
\]
By Proposition~\ref{prop:nilrad-postlie}, $\fn$ is an ideal of the associated
post-Lie algebra. Since $(\g,[\, ,\,]_{\cdot})$ is solvable, one has
$[\g,\g]_{\cdot}\subseteq \fn$.
Moreover, the proof of Proposition~\ref{prop:nilrad-postlie} gives
$\g\triangleright \g\subseteq \fn$.
Hence, using
\[
[x,y]_{\circ}=[x,y]_{\cdot}+x\triangleright y-y\triangleright x,
\]
we obtain
$[\g,\g]_{\circ}\subseteq \fn$.
Therefore the quotient post-Lie algebra
$\g/\fn$
is trivial.

If $\fn\neq \g$, then Proposition~\ref{prop:trivial-quotient-closed-ideal}
yields a nonzero proper connected closed ideal of $(G,\cdot,\circ)$,
contradicting simplicity. Hence $\fn=\g$. Therefore
$(\g,[\, ,\,]_{\cdot})$ is nilpotent, and so is the connected Lie group
$(G,\cdot)$.
\end{proof}

The previous corollary shows that, under simplicity, solvability of the
additive Lie group already forces nilpotency. This reduces the proof of the
main theorem to a nilpotent--nilpotent situation at the infinitesimal level,
which we now analyze.

\subsection{The nilpotent--nilpotent case}\label{sec:nil-nil}

We now assume that both Lie algebras underlying the post-Lie structure are
nilpotent. In this setting, results on simply transitive NIL-affine actions
imply a triangular form for the left multiplication operators. This in turn
produces a proper post-Lie ideal with trivial quotient.

\begin{lemma}\label{lem:triangular-postlie}
Let
$(\g,[\, ,\,]_{\cdot},[\, ,\,]_{\circ},\triangleright)$
be an integrable finite-dimensional real post-Lie algebra such that both Lie
algebras
\[
(\g,[\, ,\,]_{\cdot})
\qquad\text{and}\qquad
(\g,[\, ,\,]_{\circ})
\]
are nilpotent. Then there exists a basis
$A_1,\dots,A_n$
of $\g$ such that, for every $i$,
$\g_i:=\langle A_i,\dots,A_n\rangle$
is an ideal of $(\g,[\, ,\,]_{\cdot})$, and
\[
L_x(\g_i)\subseteq \g_{i+1}
\qquad \forall\, x\in \g,
\]
where $\g_{n+1}=0$.
\end{lemma}

\begin{proof}
Choose a connected simply connected Lie skew brace
$(G,\cdot,\circ)$
integrating the given post-Lie algebra. Since both Lie algebras are nilpotent,
the Lie groups $(G,\cdot)$ and $(G,\circ)$ are nilpotent as well. Hence the
associated action of $(G,\circ)$ on $(G,\cdot)$ is a simply transitive
NIL-affine action.
By \cite[Prop.~2.11, Prop.~2.12, Th.~2.15]{Burde2012affineaction}, the given
post-Lie structure corresponds to a complete NIL-affine structure
\[
d\rho:\bigl(\g,[\, ,\,]_{\circ}\bigr)\to
\bigl(\g,[\, ,\,]_{\cdot}\bigr)\rtimes \Der(\g,[\, ,\,]_{\cdot}),
\qquad
x\mapsto (x,L_x).
\]
By the proof of \cite[Th.~3.1]{BurdeDekimpeDeschamps}, there exists a basis
$A_1,\dots,A_n$
of $\g$ such that, for every $i$,
$\g_i:=\langle A_i,\dots,A_n\rangle$
is an ideal of $(\g,[\, ,\,]_{\cdot})$, and
\[
L_x(\g_i)\subseteq \g_{i+1}
\qquad \forall\, x\in \g.
\]
\end{proof}

The triangular form above immediately yields the desired quotient.

\begin{proposition}\label{prop:nil-nil}
Let
$(\g,[\, ,\,]_{\cdot},[\, ,\,]_{\circ},\triangleright)$
be an integrable finite-dimensional real post-Lie algebra such that both Lie
algebras
\[
(\g,[\, ,\,]_{\cdot})
\qquad\text{and}\qquad
(\g,[\, ,\,]_{\circ})
\]
are nilpotent. Assume that $\dim\g>1$. Then there exists a nonzero proper
ideal
$\ii\subsetneq \g$
of the post-Lie algebra such that the quotient post-Lie algebra
$\g/\ii$
is trivial.
\end{proposition}

\begin{proof}
By Lemma~\ref{lem:triangular-postlie}, choose a basis
$A_1,\dots,A_n$
of $\g$ such that, for every $i$,
$\g_i:=\langle A_i,\dots,A_n\rangle$
is an ideal of $(\g,[\, ,\,]_{\cdot})$, and
\[
L_x(\g_i)\subseteq \g_{i+1},
\qquad \forall\, x\in \g.
\]
Set
$\ii:=\g_2$.
Since $\dim\g>1$, the subspace $\ii$ has codimension one, hence
\[
0\neq \ii\neq \g.
\]

Moreover, $\ii$ is an ideal of $(\g,[\, ,\,]_{\cdot})$, and taking $i=1$ in
the inclusion above gives
\[
L_x(\g)\subseteq \g_2=\ii,
\qquad \forall\, x\in \g,
\]
that is,
\[
\g\triangleright \g\subseteq \ii.
\]
Therefore
\[
\g\triangleright \ii\subseteq \ii
\qquad\text{and}\qquad
\ii\triangleright \g\subseteq \ii,
\]
so $\ii$ is an ideal of the post-Lie algebra.
Finally, since $\ii$ has codimension one, the quotient $\g/\ii$ is
one-dimensional. Hence both quotient Lie algebras
\[
(\g/\ii,[\, ,\,]_{\cdot})
\qquad\text{and}\qquad
(\g/\ii,[\, ,\,]_{\circ})
\]
are abelian. Together with
$\g\triangleright \g\subseteq \ii$,
this shows that the induced post-Lie product on $\g/\ii$ vanishes. Thus the
quotient post-Lie algebra $\g/\ii$ is trivial.
\end{proof}

\section{Proof of the main theorem}\label{sec:proof-main}

We now combine the previous ingredients.

\begin{proof}[Proof of Theorem~\ref{thm:main}]
Let
$(\g,[\, ,\,]_{\cdot},[\, ,\,]_{\circ},\triangleright)$
be the post-Lie algebra associated with $(G,\cdot,\circ)$.
Since $(G,\circ)$ is nilpotent and $G$ is connected, it follows from
\cite[Th.~1.1, property~(R1)]{DameleLoi2026} that the additive Lie group
$(G,\cdot)$ is solvable. Assume that $(G,\cdot,\circ)$ is simple. Then
Corollary~\ref{cor:simple-brace-solvable-additive} shows that the additive Lie
group $(G,\cdot)$ is nilpotent as well. Hence both Lie algebras
\[
(\g,[\, ,\,]_{\cdot})
\qquad\text{and}\qquad
(\g,[\, ,\,]_{\circ})
\]
are nilpotent.
Assume now that $\dim G>1$. Then $\dim\g>1$, so
Proposition~\ref{prop:nil-nil} yields a nonzero proper ideal
$\ii\subsetneq \g$
of the post-Lie algebra such that the quotient post-Lie algebra $\g/\ii$ is
trivial. By Proposition~\ref{prop:trivial-quotient-closed-ideal}, this gives a
nonzero proper connected closed ideal of $(G,\cdot,\circ)$, contradicting
simplicity. Therefore one must have
$\dim G=1$.
Since both underlying Lie groups are connected and one-dimensional, they are
abelian. Therefore Lemma~\ref{lem:1dim-trivial} applies and shows that
$(G,\cdot,\circ)$ is trivial. This completes the proof.
\end{proof}

\begin{corollary}\label{cor:main-equivalent}
Let $(G,\cdot,\circ)$ be a simply connected Lie skew brace such that the
multiplicative Lie group $(G,\circ)$ is nilpotent. If
$\dim G>1$,
then $(G,\cdot,\circ)$ is not simple.
\end{corollary}

\begin{proof}
This is immediate from Theorem~\ref{thm:main}.
\end{proof}

\begin{remark}\rm
The one-dimensional case in Theorem~\ref{thm:main} is necessary. Indeed, every
connected one-dimensional abelian Lie group gives rise to a trivial simple Lie
skew brace. Under the simply connected assumption, the exceptional case is necessarily
isomorphic to $(\mathbb R,+,+)$.
\end{remark}

\end{document}